\pdfoutput=1
\documentclass[aip,jmp,notitlepage,aps,amsmath,amssymb,10pt,a4paper,english]{revtex4-1} 
\usepackage{babel}
\usepackage{amsthm}
\usepackage{hyperref}

\usepackage{csquotes}
\theoremstyle{remark}
\newtheorem{remark}{Remark}[section]
\theoremstyle{plain}
\newtheorem{theorem}[remark]{Theorem}
\newtheorem{proposition}[remark]{Proposition}
\newtheorem{lemma}[remark]{Lemma}

\theoremstyle{definition}

\let\ge=\varepsilon

\newcommand{\slap}{\left(-\Delta \right)^s}

\begin{document}
    \title{Ground state solutions for nonlinear fractional Schr\"{o}dinger equations in $\mathbb{R}^N$}\thanks{Partially supported by PRIN 2009 ``Teoria dei punti critici
e metodi perturbativi per equazioni differenziali nonlineari''.}
    \author{Simone Secchi}
    \affiliation{Dipartimento di Matematica e Applicazioni, Universit\`a di Milano-Bicocca, via Roberto Cozzi 53, I-20125 Milano}
    \email{Simone.Secchi@unimib.it}
    \homepage{http://www.matapp.unimib.it/~secchi}
    
    \date{\today} % no date

\begin{abstract}
We construct solutions to a class of Schr\"{o}dinger equations involving the fractional laplacian. Our approach is variational in nature, and based on minimization on the Nehari manifold.
\end{abstract}

\maketitle

%\vspace{5mm}

%\noindent \emph{Keywords:} Perturbation methods, pseudo-relativistic Hartree equation, fractional laplacian.

%\noindent \emph{AMS Subject Classification:} 35Q55, 35A15, 35J20

\section{Introduction}

In this paper we will study standing waves for a nonlinear differential equation driven by the fractional laplacian. We will focus on the so-called \emph{fractional Schr\"{o}dinger equation}
\begin{equation}\label{eq:1}
\mathrm{i} \frac{\partial \psi}{\partial t} = \slap \psi + V(x) \psi - \left|\psi \right|^{p-1} \psi
\end{equation}
where $(x,t) \in \mathbb{R}^N \times (0,+\infty)$, $0<s<1$, and $V \colon \mathbb{R}^N \to \mathbb{R}$ is an external potential function. The operator $\slap$ is the \emph{fractional laplacian of order} $s$, see the next section for a short review of its properties. 

This equation was introduced by Laskin (see \cite{Laskin2002,Laskin2000}), and comes from an expansion of the Feynman path integral from Brownian-like to L\'evy-like quantum mechanical paths.
When $s=1$, the L\'evy dynamics becomes the Brownian dynamics, and
equation (\ref{eq:1}) reduces to the classical Schr\"{o}dinger equation
\[
\mathrm{i} \frac{\partial \psi}{\partial t} = -\Delta \psi + V(x) \psi - \left|\psi \right|^{p-1} \psi.
\]
Standing wave solutions to this equation are solutions of the form
\begin{equation}\label{eq:2}
\psi(x,t)=\mathrm{e}^{-\mathrm{i} \omega t} u(x),
\end{equation}
where $u$ solves the elliptic equation
\[
-\Delta u + V(x) u - \left|u \right|^{p-1} u=0.
\] 
The mathematical literature for the nonlinear Schr\"{o}dinger equation is so huge that we do not even try to collect here a detailed bibliography; we only cite \cite{Rabinowitz}, from which we will borrow some ideas that have become classical over the years.
On the contrary, to the best of our knowledge, the literature for fractional Schr\"{o}dinger equations is still expanding and rather young. 

\medskip

In the sequel we will look for standing wave solutions of a more general equation than (\ref{eq:1}), and precisely we will solve
\begin{equation} \label{eq:3}
\slap u + V(x) u = f(x,u), \quad x \in \mathbb{R}^N.
\end{equation}
In \cite{FQT} Felmer \emph{et al.} studied a similar class of equations, in which $V$ is a positive constant, say $1$, and the nonlinearity $f$ satisfies suitable assumptions; roughly speaking, as $|x| \to +\infty$, $f(x,s)$ behaves like a continuous function $\bar{f}(s)$, uniformly with respect to bounded values of $s$. Using Critical Point Theory, classical positive solutions are found, and some interesting results in regularity theory are offered.

In \cite{Cheng}, the potential $V$ is allowed to vary, but the nonlinearity is a pure power $f(s)=|s|^{p-1}s$. Ground states are found by imposing a \emph{coercivity} assumption on $V$, i.e.
\[
\lim_{|x| \to +\infty} V(x) = +\infty.
\]
In \cite{DPV}, the authors look for \emph{radially symmetric} solutions of (\ref{eq:3}) when $V$ and $f$ do not depend explicitly on the space variable~$x$.

In the very particular case of dimension $N=1$, much more can be said for the autonomous equation 
\[
\slap u + u = |u|^{p-1}u \quad\text{in $\mathbb{R}$},
\]
and we refer to \cite{FL}. However, the techniques of that paper cannot be easily adapted to the same equation in general dimension, since they depend strongly on the separation of eigenvalues of the linearized operator.

\medskip

We will first provide a generalization of the main result of \cite{Cheng} for the general equation (\ref{eq:3}), and then we will give some existence result under weaker assumptions on the behavior of the potential $V$ at infinity. Finally, by exploiting some non-trivial tools for the fractional laplacian, we will solve a \emph{perturbed} equation resembling (\ref{eq:3}).

\section{Some background on the fractional laplacian}

In this section we collect some information to be used in the paper. We will denote either by $\hat{u}$ or by $\mathcal{F}u$ the usual Fourier transform of $u$. For the sake of simplicity, integrals over the whole $\mathbb{R}^N$ will be often written $\int$.

Sobolev spaces of fractional order are the convenient setting for our equations. 
A very complete introduction to fractional Sobolev spaces can be found in \cite{DNPV}; we offer below a short review.

We recall that the fractional Sobolev space $W^{s,p}(\mathbb{R}^N)$ is defined for any $p \in [1,+\infty)$ and $s \in (0,1)$ as
\[
W^{s,p}(\mathbb{R}^N) = \left\{ u \in L^p(\mathbb{R}^N) \mid \int \frac{|u(x)-u(y)|^p}{|x-y|^{sp + N}}\, dx \, dy < \infty \right\}.
\]
This space is endowed with
the Gagliardo norm
\[
\|u\|_{W^{s,p}} = \left( \int |u|^p \, dx + \int \frac{|u(x)-u(y)|^p}{|x-y|^{sp + N}}\, dx \, dy \right)^{\frac{1}{p}}.
\]
When $p=2$, these spaces are also denoted by $H^s(\mathbb{R}^N)$.

If $p=2$, an equivalent definition of fractional Sobolev spaces is possible, based on Fourier analysis. Indeed, it turns out that
\[
W^{s,2}(\mathbb{R}^N) = \left\{ u \in L^2 (\mathbb{R}^N) \mid \int \left( 1 + |\xi|^{2s} \right) |\hat{u}(\xi)|^2 \, d\xi < \infty \right\},
\]
and the norm can be equivalently written
\[
\|u\|_{W^{s,2}} = \sqrt{\|u\|_{L^2}^2 + \int |\xi|^{2s} |\hat{u}(\xi)|^2 \, d\xi}.
\]
The homogeneous fractional Sobolev space $D^{s,2}(\mathbb{R}^N)$, also denoted by $H_0^2(\mathbb{R}^N)$ or by $\dot{H}^s(\mathbb{R}^N)$,  is defined as the completion of $C_0^\infty(\mathbb{R}^N)$ with respect to the norm
\[
\|u\|_{W_0^{s,p}}^2=\int |\xi|^{2s} |\hat{u}(\xi)|^2\, d\xi;
\]
it can also be characterized as the space
\[
D^{s,2}(\mathbb{R}^N) = \left\{ u\in L^{2^*}(\mathbb{R}^N) \mid |\xi|^{s/2}\hat{u}(\xi) \in L^2 (\mathbb{R}^N) \right\}.
\]
%\footnote{$W_0^{s,p}(\mathbb{R}^N)$ is another popular notation, but we prefer the symbol $D^{s,p}(\mathbb{R}^N)$, which is commonly used when $s \in \mathbb{N}$.} 

\medskip

The fractional laplacian $\slap$ of a rapidly decaying test function $u$ is defined as
\[
\slap u (x) = C_{N,s} P.V. \int \frac{u(x)-u(y)}{|x-y|^{N+2s}}\, dx\, dy,
\]
where P.V. denotes the \emph{principal value} of the singular integral, and 
\[
C_{N,s}^{-1} = \int \frac{1-\cos x_1}{|x|^{N+2s}}\, dx.
\]
It is possible to prove that $\slap$ is a pseudo-differential operator, and more precisely that
\begin{equation}\label{eq:slap}
\slap u = \mathcal{F}^{-1} \left(|\xi|^{2s} \hat{u}(\xi) \right).
\end{equation}
In particular, the \emph{symbol} of $\slap$ is $|\xi|^{2s}$. In this paper we will mainly use (\ref{eq:slap}) as the definition of the fractional laplacian. It is also useful to remark that an equivalent norm on $W^{s,2}$ is given by
\[
\sqrt{\|u\|_{L^2}^2 + \|(-\Delta)^{\frac{s}{2}}u\|_{L^2}^2}.
\]
For the reader's convenience, we review the main embedding result for fractional Sobolev spaces.
\begin{proposition}[Sobolev embedding theorem] \label{prop:Sob}
Let $s \in (0,1)$ and $p \in [1,+\infty)$ such that $sp<N$. Then there exists a constant $C$, depending only on $N$, $s$ and $p$, such that
\[
\|u\|_{L^{p^*}}^p \leq C \int \frac{|u(x)-u(y)|^p}{|x-y|^{N+sp}} dx\, dy
\]
for every $u \in W^{s,p}(\mathbb{R}^N)$, where
\[
p^* = \frac{Np}{N-sp}
\]
is the \emph{fractional critical exponent}. Hence the embedding $W^{s,p}(\mathbb{R}^N) \subset L^q(\mathbb{R}^N)$ is continuous for any $q \in [p,p^*]$, and  compact for any $q \in [p,p^*)$.
\end{proposition}

\begin{remark}
To save notation, and since we will always work in $\mathbb{R}^N$, we will often write $W^{s,2}$, $L^p$, etc. instead of $W^{s,2}(\mathbb{R}^N)$, $L^p(\mathbb{R}^N)$, etc. 
\end{remark}

The standard \emph{Sobolev-Gagliardo-Nirenberg} inequality can also be proved in fractional spaces. We will use it in the following form.
\begin{proposition}\label{prop:GN}
Let $q>1$. Then there exists a constant $C>0$ such that
\[
\|u\|_{q+1}^{q+1} \leq C \|u\|_{W^{s,2}}^{\frac{(q-1)N}{2s}} \|u\|_2^{q+1-\frac{(q-1)N}{2s}}
\]
for every $u \in W^{s,2}$.
\end{proposition}
One major tool in variational methods is the following \emph{vanishing lemma}, originally proved by P.L. Lions.
\begin{lemma}\label{lem:lions}
Assume $\{u_k\}$ is a bounded sequence in $W^{s,2}$ which satisfies
\[
\lim_{k \to +\infty} \sup_{y \in \mathbb{R}^N} \int_{B(y,R)} |u_k(x)|^2 \, dx =0,
\]
for some $R>0$. Then $u_k \to 0$ strongly in $L^q$, for every $2<q<\frac{2N}{N-2s}$.
\end{lemma}
\begin{proof}
Pick $q \in (2,\frac{2N}{N-2s})$. Given $R>0$ and $y \in \mathbb{R}^N$, by standard interpolation and Proposition \ref{prop:Sob} we obtain
\[
\|u_k\|_{L^q(B(y,R))} \leq \|u_k\|_{L^2(B(y,R))}^{1-\lambda} \|u_k\|_{L^{\frac{2N}{N-2s}}(B(y,R))}^\lambda,
\]
where
\[
\frac{1-\lambda}{2} + \frac{\lambda}{\frac{2N}{N-2s}} = \frac{1}{q}.
\]
Using a locally finite covering of $\mathbb{R}^N$ consisting of balls of radius $R$, we deduce that
\[
\|u_k\|_q \leq C \|u_k\|_{L^2(B(y,R))}^{(1-\lambda)q} \|u_k\|_{\frac{2N}{N-2s}}^{\lambda q},
\]
and we conclude by the Sobolev embedding theorem \ref{prop:Sob}.
\end{proof}

In the last section we will need to use cut-offs; it is clear that this technique, which is very useful for local operators, may become troublesome for non-local operators like the fractional laplacian. Indeed, these operators cannot be easily localized, and moreover the fractional laplacian of a product does not satisfy, in general, Leibnitz's rule of differentiation.

However, in some very special cases, there are workarounds.
\begin{lemma} \label{lem:Palatucci}
Suppose that $0<s<N/2$, and let $u \in D^{s,2}(\mathbb{R}^N)$. If $\varphi \in C_0^\infty(\mathbb{R}^N)$ and $\phi_\lambda (x) = \phi(x / \lambda)$, then $\phi_\lambda u \to 0$ in $D^{s,2}(\mathbb{R}^N)$ as $\lambda \to 0$.

Moreover, if $\phi=1$ on a neighborhood of zero, then $\phi_\lambda u \to u$ in $D^{s,2}(\mathbb{R}^N)$ as $\lambda \to +\infty$.
\end{lemma}
We refer to \cite[Lemma 4.1]{Palatucci} for the proof, which is not elementary at all, and requires properties of multipliers between Sobolev spaces.

The next lemma provides a way to manipulate, in some cases, smooth truncations for the fractional laplacian.
\begin{lemma} \label{lem:Palatucci-bis}
Suppose that $0<s<N/2$, and let $\varphi \in W^{\sigma,2}(\mathbb{R}^N)$ for $\sigma > 1+N/2$. Then the commutator $[\varphi , (-\Delta)^{\frac{s}{2}}] \colon D^{s,2}(\mathbb{R}^N) \to L^2(\mathbb{R}^N)$ is continuous, i.e. 
\[
\varphi \cdot \left( (-\Delta)^{\frac{s}{2}}u_n) \right)- (-\Delta)^{\frac{s}{2}} \left( \varphi u_n \right) \to 0 \quad \text{in $L^2(\mathbb{R}^N)$}
\]
whenever $u_n \to 0$ in $D^{s,2}(\mathbb{R}^N)$.
\end{lemma}
\begin{proof}
For the sake of convenience, we provisionally write $L=(-\Delta)^{\frac{s}{2}}$. For each $\ge>0$ we set $L_\ge = \left(\ge I - \Delta \right)^{\frac{s}{2}}$, where $I$ is the identity operator. It is clear that
\begin{align*}
Lu &= \mathcal{F}^{-1} \circ M_{|\xi|^s} \circ \mathcal{F} \\
L_\ge u &= \mathcal{F}^{-1} \circ M_{\left( |\xi|^2 + \ge \right)^{\frac{s}{2}}} \circ \mathcal{F},
\end{align*}
where $M_{P(\xi)}$ is the multiplication operator with symbol $P$. The operator $L_\ge \colon W^{s,2}(\mathbb{R}^N) \to L^2(\mathbb{R}^N)$ is therefore a bounded operator, and similarly $L \colon W^{s,2}(\mathbb{R}^N) \to L^2 (\mathbb{R}^N)$. The operator norm of $L_\ge-L$ can be easily estimated:
\[
\|L_\ge-L\|_{\mathcal{L}(W^{s,2}(\mathbb{R}^N),L^2 (\mathbb{R}^N))} \leq \sup_{\xi \in \mathbb{R}^N} \frac{\left| \left(\ge + |\xi|^2 \right)^{\frac{s}{2}}- |\xi|^s \right|}{\left( 1 + |\xi|^2 \right)^{\frac{s}{2}}},
\]
so that $L_\ge \to L$ in the operator norm as $\ge \to 0$. It is therefore sufficient to prove that $[L_\ge,\varphi]$ is continuous. Since $0<s<N/2$ and $\varphi$ belongs, in particular, to $W^{\sigma ,2}(\mathbb{R}^N)$ with $\sigma > N/2+1$, we can apply \cite[Proposition 4.2]{Taylor}, which states that, for some constant $C>0$,
\[
\left\| [L_\ge,\varphi ] u \right\|_{L^2} \leq C\|\varphi\|_{W^{\sigma,2}} \|u\|_{W^{s-1,2}}.
\]
Since the embedding $D^{s,2}(\mathbb{R}^N) \to W^{s-1,2}(\mathbb{R}^N)$ is continuous, the proof is complete.
\end{proof}
\begin{remark}
Our proof follows that of~\cite[Lemma 4.2]{Palatucci}. Since we work in $\mathbb{R}^N$, however, we cannot expect the commutator to be also \emph{completely continuous} as in that reference.
\end{remark}

\section{A variational setting}

We introduce the Hilbert space
\[
E^s = \left\{ u \in L^2 \mid \int |\xi|^{2s} \left| \hat{u}(\xi)\right|^2\, d\xi + \int V(x) \left| u(x) \right|^2 \, dx < \infty
\right\}
\]
endowed with the inner product
\[
\langle u,v \rangle = \int |\xi|^{2s} \hat{u}(\xi) \hat{v}(\xi) \, d\xi + \int V(x) u(x)v(x) \, dx
\]
and its associated norm.
We shall always assume
\begin{itemize}
\item[(V$_1$)] $V \in C^1 (\mathbb{R}^N)$ and $\inf_{x \in \mathbb{R}^N} V(x) =V_0 > 0.$
\end{itemize}
Moreover, the following assumptions on the non-linearity $f=f(x,s)$ will be retained:
\begin{itemize}
\item[($f_1$)] $f \in C^1 (\mathbb{R}^N\times \mathbb{R})$;
\item[($f_2$)] $f(x,0)=0=\frac{\partial f}{\partial s}(x,0)$ for every $x \in \mathbb{R}^N$;
\item[($f_3$)] there are constants $a_1$, $a_2>0$ and $1 < p < \frac{N+2s}{N-2s}$ such that 
\[
\left| \frac{\partial f}{\partial s} (x,s) \right| \leq a_1 + a_2 \left| s \right|^{p-1},
\]
for every $x \in \mathbb{R}^N$ and $s \in \mathbb{R}$;
\item[($f_4$)] (Ambrosetti-Rabinowitz condition) there is a constant $\mu >2$ such that
\[
0< \mu F(x,s) \leq s f(x,s)
\]
for all $x \in \mathbb{R}^N$ and $s \in \mathbb{R} \setminus \{0\}$. Here
\[
F(x,s) = \int_0^s f(x,t)\, dt.
\]
\end{itemize}
Weak solutions to (\ref{eq:3}) are critical points of the functional $J \colon E^s \to \mathbb{R}$ defined by
\begin{equation}\label{eq:J}
J(u) = \frac{1}{2} \int |\xi|^{2s} \left| \hat{u}(\xi) \right|^2 \, d\xi + \frac{1}{2} \int V(x) \left|u(x)\right|^2 \, dx - \int F(x,u(x))\, dx.
\end{equation}
It is a simple exercise to check that $J$ is well-defined and of class $C^1$, as a consequence of our assumptions on $f$. Moreover, $E^s$ is continuously embedded in $W^{s,2}(\mathbb{R}^N)$, due to assumption (V$_1$).

The functional $J$ has the mountain-pass geometry, and we can introduce the following class of paths:
\[
\Gamma = \left\{ g \in C([0,1],E^s) \mid g(0)=0,\ J(g(1))<0 \right\}.
\]
The mountain-pass level
\[
c = \inf_{g \in \Gamma} \sup_{0 \leq t \leq 1} J(g(t))>0
\]
is therefore associated to $\Gamma$. Since the non-compact group of translations acts on $\mathbb{R}^N$, we cannot expect the Palais-Smale condition to be satisfied by $J$ without further assumptions. This forces us to analyze this lack of compactness.

As a first result, we obtain the existence of a non-trivial solution under a coercivity assumption on $V$. This theorem appears in \cite{Cheng} in the special case $f(x,u)=|u|^{p-1}u$.

\begin{theorem} \label{th:1}
Retain assumptions (V$_1$) and ($f_1$)--($f_4$). Assume moreover that 
\[
\lim_{|x| \to +\infty} V(x)=+\infty.
\] 
Then equation (\ref{eq:3}) has at least a non-trivial solution $u$.
\end{theorem}
\begin{proof}
The strategy of the proof is simple: we need to check that the mountain-pass level $c$ is a critical value for $J$. By Ekeland's variational principle, there exists a sequence $\{u_m\}$ in $E^s$ such that
\begin{equation*}
\lim_{m \to +\infty} J(u_m)=c, \quad \text{$\lim_{m \to +\infty} DJ(u_m)=0$ strongly.}
\end{equation*}
The usual estimate
\begin{eqnarray*}
c+1+\|u_m\| &\geq& J(u_m)-\mu^{-1} DJ(u_m)u_m \\
&=& \left(\frac{1}{2} - \frac{1}{\mu} \right) \|u_m\|^2 + \int \left[ \mu^{-1} f(x,u_m(x))-F(x,u_m(x))\right] \, dx \\
&\geq& \left(\frac{1}{2} - \frac{1}{\mu} \right) \|u_m\|^2
\end{eqnarray*}
implies the boundedness of $\{u_m\}$ in $E^s$. Up to a subsequence, we may suppose that $u_m \to u$ weakly in $E^s$ and locally strongly in $L^q$, for $1 \leq q < 2N/(N-2s)$. Clearly $u \in E^s$ is a weak solution to (\ref{eq:3}), but it may happen that $u=0$. To exclude this possibility, we remark that, at least for $m \gg 1$,
\begin{eqnarray} \label{eq:5}
\frac{c}{2} &\leq& J(u_m)-\frac{1}{2}DJ(u_m)u_m \notag \\
&=& \int \left[ \frac{1}{2}f(x,u_m(x))u_m(x)\, dx - F(x,u_m(x)) \right] \, dx.
\end{eqnarray}
Pick $\ge>0$ and a constant $A_\ge>0$, depending on $p$ and $\ge$, such that
\[
\left| \frac{\partial f}{\partial s}(x,s) \right| \leq \ge + A_\ge |s|^{p-1}
\]
for every $x \in \mathbb{R}^N$ and $s \in \mathbb{R}$. Integrating this inequality we get
\[
\left| f(x,s) \right| \leq \ge |s| + \frac{1}{p} A_\ge |s|^p.
\]
Inserting into (\ref{eq:5}) we find
\[
\frac{c}{2} \leq \int \left( \frac{\ge}{2} |u_m(x)|^2 + \frac{A_\ge}{2p} |u_m(x)|^{p+1} \right)dx.
\]
Using Proposition \ref{prop:GN} we can write, for some constant $C_1>0$:
\[
\frac{c}{2} \leq \frac{\ge}{2} \|u_m\|_2^2 + C_1 
\|u\|_{W^{s,2}}^{\frac{(p-1)N}{2s}} \|u\|_2^{p+1-\frac{(p-1)N}{2s}}
\]
Choose now 
\[
\ge \leq \frac{c}{2 \left( \sup_m \|u_m\| \right)^2},
\]
and notice that this yields
\[
\|u_m\|_2 \geq  \exp \left( \frac{C_2}{p+1-\frac{(p-1)N}{2s}} \log \frac{c}{4} \right) \equiv \tilde{c}.
\]
%\[
%\|u_m\|_2 \geq C_2 
%\left( \frac{c}{4} \right)^{\frac{1}{p+1-\frac{(p-1)N}{2s}}} \equiv \tilde{c}.
%\]
Assume, by contradiction, that the weak limit $u$ is zero. Then, for every $R>0$ there is some $m_0=m_0(R)$ with the property that, for any $m \geq m_0$,
\[
C_2 \|u_m\|_{L^2(B(0,R))} \leq \frac{\tilde{c}}{2}.
\]
But then
\begin{eqnarray} \label{eq:6}
\frac{\tilde{c}}{2} &\leq& C_2 \|u_m\|_{L^2(\mathbb{R}^N \setminus (0,R))} \notag \\
&\leq& \frac{C_2}{\inf_{|x| \geq R} \sqrt{V(x)}} \sqrt{\int_{|x|\geq R} V(x) u_m(x)^2\, dx} \notag \\
&\leq& \frac{C_2 \left( \sup_m \|u_m\| \right)}{\inf_{|x| \geq R} \sqrt{V(x)}}.
\end{eqnarray}
Since $V(x) \to +\infty$ as $|x| \to +\infty$, we reach a contradiction when $R \gg 1$. Therefore the weak limit $u$ is not trivial, and the proof is complete.
\end{proof}
\begin{remark}
The proof shows that the coercivity of $V$ may be relaxed. Indeed, the last line of (\ref{eq:6}) shows that we need a quantitative estimate on $\sqrt{V}$ at infinity. This is possible, since the constant $C_2$ can be expressed in terms of the constants appearing in our assumptions.
\end{remark}
\begin{remark}
A different proof can be supplied, by using the fact that $E^s$ is compactly embedded into $L^{p+1}$ when $V$ is coercive (see for instance \cite[Lemma 3.2]{Cheng}). The Palais-Smale sequence $\{u_m\}$ converges therefore \emph{strongly} in $L^{p+1}$, and (\ref{eq:5}) implies that
\[
\frac{c}{2} \leq \int \left[ f(x,u(x))u(x)-F(x,u(x)) \right]dx.
\] 
If $u=0$, then $c=0$, a contradiction. This proof, anyway, is not quantitative, and the coercivity assumption on $V$ cannot be easily relaxed to a suitable ``largeness'' condition, as in our previous remark.
\end{remark}
\begin{remark}
Our assumptions guarantee that weak solutions have higher regularity properties, i.e. they are H\"{o}lder continuous, they satisfy (\ref{eq:3}) pointwise, and decay to zero at infinity. We refer to Section~\ref{subs:5.1} for a summary of the regularity properties. For a different approach, we refer to \cite{FQT} and \cite{Cheng}. 

We will tacitly make use of these facts, and in particular of the continuity of our weak solutions, in the following sections.
\end{remark}
It is also possible to find \emph{positive} solutions of (\ref{eq:3}). We sketch the ideas. First of all, we replace $f$ by 
\[
f^{+}(x,s) = 
\begin{cases}
f(x,s) &\text{if $s \geq 0$} \\
0 &\text{otherwise.}
\end{cases}
\]
Theorem \ref{th:1} carries over to the equation in which $f$ is replaced by $f^{+}$, and we get a weak solution $u^{+}$ of the equation
\[
\slap u^{+} + V(x) u^{+} = f^{+}(x,u^{+}). 
\]
If $u^{+}$ becomes negative somewhere, we consider the set $\mathcal{O} = \{ x \mid u^{+}(x)<0 \}$ and we apply the maximum principle for the fractional laplacian (which holds true for semicontinous solutions, see \cite[2.2.28]{Silvestre}) on $\mathcal{O}$. Then $\mathcal{O}=\emptyset$, and in fact $u^{+}>0$ everywhere.

\section{The Nehari manifold and qualitative properties of ground-state levels}

The coercivity assumption for $V$ is rather strong, and we may wonder if we can relax it. We will show that non-trivial solutions of (\ref{eq:3}) exist under weaker assumptions, but we need a different variational approach. We introduce the \emph{Nehari manifold} associated to $J$ as follows:
\[
\mathcal{N}^s = \left\{ u \in E^s \setminus \{0\} \mid \int |\xi|^{2s} |\hat{u}(\xi)|^2 \, d\xi + \int V(x)|u(x)|^2\, dx = \int f(x,u(x))u(x)\, dx \right\}.
\]
As in the case of the standard laplacian, we have a topological structure on $\mathcal{N}^s$ under some additional assumption on $f$, i.e.
\begin{itemize}
\item[($f_5$)] The map $t \mapsto t^{-1} s f(x,ts)$ is increasing on $(0,+\infty)$, for every $x \in \mathbb{R}^N$ and $s \in \mathbb{R}$.
\end{itemize}

\begin{lemma}
Besides our standing assumptions, retain also ($f_5$). The Nehari ma\-ni\-fold $\mathcal{N}^s$ is then non-empty, and it is radially homeomorphic to the unit sphere of $E^s$.
\end{lemma}
\begin{proof}
Fix any $\psi \in E^s$, and consider the path $t \mapsto t \psi$. Now,
\[
J(t\psi) = \frac{t^2}{2} \|u\|^2 - \int F(x,t\psi(x))\, dx.
\]
The assumptions on $f$ imply that the last term is super quadratic when $t \gg 1$, so that $t \mapsto J(t\psi)$ attains a \emph{unique} (because of ($f_5$)) global maximum at some $t = \phi (u) > 0$. Differentiating, we find
\[
0=\left. \frac{d}{dt}J(t\psi) \right|_{t=\phi(u)} = \phi(u) \|u\|^2 - \int f(x,\phi(u)u(x))u(x)\, dx,
\]
and therefore $\phi(u)u \in \mathcal{N}^s$. Hence $\mathcal{N}^s \neq \emptyset$. 
We can prove now that $u \mapsto \phi(u)$ is a continuous map from $E^s \setminus \{0\} \to (0,+\infty)$, which implies that $\mathcal{N}^s$ is radially homeomorphic to the unit sphere of $E^s$. 

To this aim, suppose $u_m \to u$ in $E^s \setminus \{0\}$. By definition,
\begin{equation}\label{eq:7}
\phi(u_m)^2 \|u_m\|^2 = \int f(x,\phi(u_m)u_m(x))\phi(u_m)u_m(x)\, dx,
\end{equation}
and either (i) $\phi(u_m) \leq 1$ or (ii) $\phi(u_m)>1$. If case (ii) prevails, then 
\begin{eqnarray*}
\int f(x,\phi(u_m)u_m(x))\phi(u_m)u_m(x)\, dx &\geq& \mu \int F(x,\phi(u_m)u_m(x))\, dx \\
&\geq& \mu \int \phi(u_m)^\mu F(x,u_m(x))\, dx.
\end{eqnarray*}
Putting together these facts,
\[
\phi(u_m)^{\mu-2} \leq \mu^{-1} \frac{\|u_m\|^2}{\int F(x,u_m(x))\, dx} \xrightarrow{m \to +\infty} \mu^{-1} \frac{\|u\|^2}{\int F(x,u(x))\, dx}.
\]
Hence $\{\phi(u_m)\}_m$ is bounded from above, and 
a subsequence of $\{\phi(u_m)\}$ converges to $\phi_\infty$, and (\ref{eq:7}) shows that $\phi_\infty=0$ implies $u=0$. Since $u \neq 0$, $\phi(u_m) \to \phi_\infty \neq 0$, and again (\ref{eq:7}) shows that
\[
\phi_\infty^2 \|u\|^2 = \int f(x,\phi_\infty u(x))\phi_\infty u(x)\, dx.
\]
This means that $\phi_\infty u \in \mathcal{N}^s$, and, by uniqueness, $\phi_\infty = \phi(u)$. We have proved that the sequence $\{\phi(u_m)\}$ has in any case a convergent subsequence, and the limit is independent of the subsequence itself. Therefore the \emph{whole} sequence $\{\phi(u_m)\}$ converges to $\phi(u)$, and the proof is complete.
\end{proof}

In the sequel, we will need to estimate the behavior of $J$ on $\mathcal{N}^s$. The following identities will be useful.

\begin{lemma} \label{lem:4.2}
Define
\[
c^\star = \inf_{u \in E^s \setminus \{0\}} \max_{\theta \geq 0} J(\theta u).
\]
Then
\[
c^\star = c = \inf_{u\in \mathcal{N}^s} J(u).
\]
\end{lemma}
\begin{proof}
The proof is rather standard. The identity $c^\star = \inf_{\mathcal{N}}J$ is a trivial consequence of the previous Lemma. To prove that $c=\inf_{\mathcal{N}}J$, we fix an arbitrary $u \in \mathcal{N}^s$ and define a path $g_u$ as follows: $g_u(t)=t Tu$, where $J(Tu)<0$. Since $g_u \in \Gamma$, $c \leq \inf_{\mathcal{N}}J$. On the other hand, if $g \in \Gamma$, then $g(t) \in \mathcal{N}^s$ for some $t \in (0,1)$. Indeed, if $DJ(g(t))g(t)>0$, then $J(g(t)) \geq 0$ for every $t$, and this contradicts the fact that $J(g(1))<0$.
\end{proof}
In the rest of this section, we will study some qualitative properties of the level $c$ as a function of the potential $V$. For this reason, we introduce the provisional notation $c(V)$ for $c$.

\begin{proposition}
Retain assumptions (V$_1$), ($f_1$--$f_5$). Let $\tilde{V}$ be a second potential, verifying (V$_1$). If $V \geq \tilde{V}$, then $c(V) \geq c(\tilde{V})$.
\end{proposition}
\begin{proof}
To prove this monotonicity property of $c$,  we first introduce the functional 
\[
\tilde{J}(u) = \frac{1}{2}\int |\xi|^{2s}|\hat{u}(\xi)|^2 \, d\xi + \frac{1}{2}\int \tilde{V}(x)|u(x)|^2 \, dx - \int F(x,u(x))\, dx
\]
associated to the potential $\tilde{V}$. Clearly $J(u) \geq \tilde{J}(u)$ at any $u \in E^s$. Let $\tilde{\Gamma}$ be the analogue of $\Gamma$ for $\tilde{J}$. If $g \in \Gamma$, then $\tilde{g} \in \tilde{\Gamma}$, and
\[
\max_{0 \leq t \leq 1} J(g(t)) \geq \max_{0 \leq t \leq 1} \tilde{J}(g(t)).
\]
Minimizing with respect to $g$ gives
\[
c \geq \inf_{\gamma\in \tilde{\Gamma}} \max_{0 \leq t \leq 1} \tilde{J}(g(t)) = \tilde{c}. \qedhere
\]
\end{proof}
This monotonicity is the key to prove the continuity of $c(V)$ with respect to $V$.
\begin{proposition} \label{prop:4.4}
Retain assumptions ($f_1$--$f_5$). Suppose that $V$ and all the potentials of a sequence $\{V_m\}$ satisfy (V$_1$). If $V_m \to V$ uniformly, then $c(V_m) \to c(V)$.
\end{proposition}
\begin{proof}
Pick $\ge>0$. For $m \gg 1$,
\[
V+\ge \geq V + |V_m-V| \geq V \geq V-|V_m-V| \geq V-\ge.
\]
By the monotonicity of $c(V)$, it is enough to prove the weaker result
\[
\lim_{\ge \to 0} c(V+\ge) = c(V).
\]
Put, to make notation lighter, $c_\ge = c(V+\ge)$. Again by monotonicity,
\[
\lim_{\ge \to 0-} c(\ge) = \underline{c} \leq c(V) = c_0.
\]
Suppose that $\underline{c}<c_0$, and consider the functional
\[
J_\ge (u) = \frac{1}{2} \int |\xi|^{2s}|\hat{u}(\xi)|^2 \, d\xi + \frac{1}{2} \int (V(x)+\ge)|u(x)|^2 \, dx - \int F(x,u(x))\, dx.
\]
Pick any sequence $\{\ge_k\}$ such that $\ge_k \to 0^{-}$ as $k \to +\infty$, and let $\delta_m \to 0^{+}$ as $m \to +\infty$. By Lemma \ref{lem:4.2}, for each $k \in \mathbb{N}$ there exists a sequence $\{u_{km}\}_m$ in $E^s$ such that $\|u_{km}\|=1$ and 
\[
\max_{\theta \geq 0} J_{\ge_k}(\theta u_{km}) \leq c_{\ge_k}+\delta_m.
\]
To each $u_{km}$ we associate a path $g_{km}$ such that
\[
\max_{0 \leq t \leq 1} J_{\ge_k}(g_{km}(t)) = \max_{\theta \geq 0} J(\theta u_{km}),
\]
as we did in Lemma \ref{lem:4.2}. A standard result in Critical Point Theorem (see for example \cite[Theorem 4.3]{MaWi}) states that there are sequences $\{w_{km}\}$ in $E^s$ and $\{t_{km}\}$ in $[0,1]$ such that 
\begin{eqnarray*}
&&\|w_{km}-g_{km}(t_{km})\| \leq \sqrt{\delta_m}, \\
&&J_{\ge_k}(w_{km}) \in \left( c_{\ge_k}-\delta_m,c_{\ge_k} \right), \\
&&\|DJ_{\ge_k}(w_{km}) \| \leq \sqrt{\delta_m}.
\end{eqnarray*}
Specializing to $m=k$, and setting $u_k=u_{kk}$, $w_k=w_{kk}$, we deduce that
\begin{eqnarray*}
c_0 &\leq& \max_{\theta \geq 0} J(\theta u_k) = J(\phi(u_k)u_k) \\
&=& J_{\ge_k} (\phi(u_k)u_k) - \ge_k \phi(u_k)^2 \|u_k\|_2^2 \\
&\leq& \max_{\theta \geq 0} J_{\ge_k}(\theta u_k) - \ge_k \phi(u_k)^2 \|u_k\|_2^2 \\
&\leq& c_{\ge_k} + \delta_k -\ge_k \phi(u_k)^2 \|u_k\|_2^2 \\
&\leq& \underline{c}+\delta_k -\ge_k \phi(u_k)^2 \|u_k\|_2^2.
\end{eqnarray*}
Since $\|u_k\|=1$, there is a constant $M_1>0$ such that $\sup_k \|u_k\|_2 \leq M_1$. 
Hence the sequence $\{\phi(u_k)\}$ cannot be bounded, otherwise the last inequalities contradict the assumption~$\underline{c}<c$. Recalling the definition of $\phi(u_k)$, we must conclude that $\phi(u_k)>1$ for large $k$, so that
\[
\phi(u_k)^2 \geq \mu \int F(x,\phi(u_k)u_k(x))\, dx \geq \mu \phi(u_k)^\mu \int F(x,u_k(x))\, dx,
\]
or
\[
\phi(u_k)^{\mu-2} \leq \frac{1}{\mu \int F(x,u_k(x))\, dx}.
\]
Since there is no upper bound for $\phi(u_k)$, the denominator must approach zero as $k \to +\infty$. But this is impossible. Indeed, the map $g_k(t)=g_{kk}(t)$ has the form (by construction, see Lemma \ref{lem:4.2}) $\psi_k(t)u_k$. The properties of $w_k$ imply now that
\[
\|w_k-\psi_k(t)u_k\| \leq \sqrt{\delta_k}.
\]
Since $\{w_k\}$ is bounded, there is a constant $M_2>0$ such that 
\[
\psi_k(t) \leq \sqrt{\delta_k} + \|w_k\| \leq M_2.
\]
For any ball $B(y,r)$, we have
\begin{align*}
\|u_k\|_{L^2(B(y,r))} &\geq M_2^{-1} \|\psi_k(t)u_k \|_{L^2(B(y,r))} \\
&\geq M_2^{-1} \left( \|w_k\|_{L^2(B(y,r))} - \|w_k - \psi_k(t)u_k \|_{L^2(B(y,r))} \right) \\
&\geq M_2^{-1} \left( \|w_k\|_{L^2(B(y,r))} - M_3 \sqrt{\delta_k} \right).
\end{align*}
By the generalized Lions' Lemma \ref{lem:lions}, there are a sequence of points $\{y_k\}$ and numbers $\beta$, $R>0$ such that
\[
\liminf_{k \to +\infty} \int_{B(y_k,R)} |w_k|^2 \geq \beta.
\]
Hence, for $k \gg 1$,
\begin{equation} \label{eq:8}
\|u_k\|_{L^2(B(y_k,R)} \geq M_2^{-1} \sqrt{\frac{\beta}{2}}.
\end{equation}
Recall that we want to prove that $\int F(x,u_k(x))\, dx \to 0$ is impossible. From ($f_4$), given $\gamma >0$, there exists $A_\gamma>0$ such that $|s|^2 \leq \gamma + A_\gamma F(x,s)$ for all $x \in \mathbb{R}^N$ and $s \in \mathbb{R}$. Consequently,
\[
\int_{B(y_k,R)} |u_k|^2 \leq \gamma + A_\gamma \int_{B(y_k,R)}F(x,u_k(x))\, dx.
\]
If $\int F(x,u_k(x))\, dx \to 0$, then $\int_{B(y_k,R)} |u_k|^2 \to 0$, contrary to (\ref{eq:8}).

We have finally proved that
\[
\lim_{\ge \to 0-} c_\ge = c_0.
\]
To complete the proof, assume by contradiction that
\[
c_0 < \overline{c}=\lim_{\ge \to 0+} c_\ge.
\]
Let $\delta_k$ be as before; again, there is a sequence $\{u_k\}$ in $E^s$ such that $\|u_k\|=1$ and 
\[
\max_{\theta \geq 0} J(\theta u_k) = c_0+\delta_k.
\]
Choose $w_k=w_{kk}$ as above, and fix $\ge>0$. Let $\phi_\ge$ be the radial homeomorphism induced by $J_\ge$, as $J$ induced $\phi$. Hence
\begin{align*}
\overline{c} &< c_\ge  \leq \max_{\theta \geq 0} J_\ge (\theta u_k) = J_\ge (\phi_\ge(u_k)u_k) \\
&= J(\phi_\ge (u_k)u_k) + \ge \phi_\ge (u_k)^2 \|u_k\|_2^2 \\
&\leq c_0 + \delta_k + \ge \phi_\ge (u_k)^2 \|u_k\|_2^2.
\end{align*}
As above, either $\phi_\ge(u_k)\leq 1$ or
\[
\phi_\ge(u_k)^{\mu-2} \leq \frac{\int |\xi|^{2s} |\hat{u}(\xi)|^2\, d\xi + \int (V(x)+\ge)|u_k(x)|^2 \, dx}{\mu \int F(x,u_k(x))\, dx}.
\]
In any case, we can conclude as earlier that $\{\phi_\ge(u_k)\}$ is a bounded sequence, and $c_0 < \overline{c}=\lim_{\ge \to 0+} c_\ge$ cannot hold. This completes the proof.
\end{proof}

\section{Existence results}

In this section we will prove some existence results for equation (\ref{eq:3}). We introduce the main assumption on the potential $V$:
\begin{itemize}
\item[(V$_2$)] for some constant $V_\infty>0$, there results
\[
\liminf_{|x| \to +\infty} V(x) \geq V_\infty.
\]
\end{itemize}
\begin{remark}
The case $V_\infty = V_0$ is not excluded.
\end{remark}
Since this assumption deals with the behavior of $V$ at infinity, it is natural to compare our equation (\ref{eq:3}) to a ``problem at infinity''; we introduce the functional $J^\infty \colon E^s \to \mathbb{R}$ by
\[
J^\infty(u) = \frac{1}{2} \int |\xi|^{2s} |\hat{u}(\xi)|^2 \, d\xi + \frac{V_\infty}{2} \int  |u(x)|^2 \, dx - \int F(x,u(x))\, dx.
\]
This functional is of class $C^1$ and has the mountain-pass geometry (see \cite{FQT}); hence we can set
\[
\Gamma^\infty = \left\{ g \in C([0,1],E^s) \mid g(0)=0,\ J^\infty (g(1))<0 \right\}
\]
and 
\[
c_\infty = \inf_{g \in \Gamma^\infty} \max_{0 \leq t \leq 1} J^\infty (g(t)).
\]
Here is a first, general, existence result.
\begin{theorem} \label{th:5.1}
Assume (V$_1$), (V$_2$) and ($f_1$--$f_5$). Then either $c$ is a critical value of $J$, or $c_\infty \leq c$.
\end{theorem}
\begin{proof}
We first prove the theorem under the stronger assumption
\begin{equation} \label{eq:9}
\liminf_{|x| \to +\infty} V(x) > V_\infty.
\end{equation}
Since the proof makes use of several techniques already presented in the previous section, we will be sketchy. As earlier, the different characterization of the level $c$ provides a sequence $\{u_m\}$ in $E^s$ such that $\|u_m\|=1$ and
\[
\max_{\theta \geq 0} J(\theta u_m) = c + o(1).
\]
Attach a path $g_m \in \Gamma$ to each $u_m$ in such a way that
\[
\max_{0 \leq t \leq 1} g_m(t) = \max_{\theta \geq 0} J(\theta u_m).
\]
Once again, we can find sequences $\{w_m\}$ in $E^s$, $\ge_m \to 0$ and $t_m \in [0,1]$ such that
\begin{align*}
&\|w_m - g_m(t_m)\| \leq \sqrt{\ge_m} \\
&J(w_m) \in (c-\ge_m,c) \\
&\|DJ(w_m)\| \leq \sqrt{\ge_m}.
\end{align*}
It follows easily that $\{w_m\}$ is bounded, and we assume that, up to subsequences, it converges weakly in $E^s$ to some $w$ and strongly in $L^q$, for any $q \in [1,\frac{2N}{N-2s})$. Then $w$ weakly solves the limiting equation
\begin{equation} \label{eq:10}
\slap w + V_\infty w = f(x,w).
\end{equation}
Lemma \ref{lem:lions} implies the existence of a sequence of points $y_m \in \mathbb{R}^N$ and of constants $\beta>0$ and $R>0$ such that
\[
\liminf_{m \to +\infty} \int_{B(y_m,R)} |w_m(x)|^2 \, dx > \beta.
\]
If the sequence $\{y_m\}$ is bounded, then $w \neq 0$ and the local compactness of the Sobolev embedding tells us that, for every $\rho>0$,
\begin{eqnarray*}
J(w_m) - \frac{1}{2}DJ(w_m)w_m &=& \frac{1}{2}\int \left( f(x,w_m(x))w_m(x) - F(x,w_m(x)) \right)dx \\
&\geq& \frac{1}{2}\int_{B(0,\rho)} \left( f(x,w_m(x))w_m(x) - F(x,w_m(x)) \right)dx \\
&=& \frac{1}{2}\int_{B(0,\rho)} \left( f(x,w(x))w(x) - F(x,w(x)) \right)dx + o(1).
\end{eqnarray*}
Letting $m \to +\infty$,
\[
c \geq \frac{1}{2}\int_{B(0,\rho)} \left( f(x,w(x))w(x) - F(x,w(x)) \right)dx.
\]
But the right-hand side of this relation coincides with $J^\infty(w)$, since $w$ solves (\ref{eq:10}), and therefore
\[
c \geq c_\infty.
\]
If, on the other hand, $\{y_m\}$ is unbounded, and we may even assume that $y_m \to +\infty$, then, for every $\alpha>0$ and $\rho>0$,
\begin{eqnarray*}
\max_{\theta \geq 0} J(\theta u_m) &\geq& J(\alpha u_m) 
=J^\infty (\alpha u_m) + \frac{1}{2}\int_{B(0,\rho)} \left( V(x)-V_\infty \right) |\alpha u_m(x)|^2\, dx \\
&&{}+ \frac{1}{2} \int_{\mathbb{R}^N \setminus B(0,\rho)} \left( V(x)-V_\infty \right) |\alpha u_m(x)|^2\, dx.
\end{eqnarray*}
Thanks to assumption~(\ref{eq:9}), we may choose $\rho>0$ so that $V(x) \geq V_\infty$ whenever $|x| \geq \rho$. Thus
\begin{equation*}
\max_{\theta \geq 0} J(\theta u_m) \geq J^\infty (\alpha u_m) +\frac{1}{2}\int_{B(0,\rho)} \left( V(x)-V_\infty \right) |\alpha u_m(x)|^2\, dx.
\end{equation*}
Specialize now $\alpha = \phi^\infty(u_m)$, where $\phi^\infty (u_m)$ is the unique positive number such that $\phi^\infty(u_m)u_m$ belongs to the Nehari manifold of $J^\infty$. As such,
\[
J^\infty (\phi^\infty(u_m)) = \max_{\theta \geq 0} J^\infty(\alpha u_m)
\]
and
\[
\max_{\theta \geq 0} J^\infty (\theta u_m) \geq c_\infty + \frac{1}{2} \int_{B(0,\rho)} \left( V(x)-V_\infty \right) |\phi^\infty (u_m)u_m(x)|^2\, dx.
\]
As earlier, $\{\phi^\infty(u_m)\}$ is a bounded sequence; if the $L^2$-norm of $u_m$ is bounded away from zero on $B(0,\rho)$, i.e. if 
\begin{equation} \label{eq:11}
\int_{B(0,\rho)}|u_m(x)|^2 \, dx \geq \gamma_1^2
\end{equation}
for some $\gamma_1>0$, then the properties of $w_m$ imply that
\begin{align*}
\|w_m\|_{L^2(B(0,\rho)} &\geq \|g_m(t_m)u_m\|_{L^2(B(0,\rho)} - \|w_m-g_m(t_m)u_m\|_{L^2(B(0,\rho)} \\
&= \|g_m(t_m)u_m\|_{L^2(B(0,\rho)} +o(1)
\end{align*}
We remark that $\{g_m(t_m)\}$ must be bounded away from zero (otherwise $g_m(t_m)u_m \to 0$ and $c +o(1) = J(g_m(t_m)u_m) = o(1)$, which is impossible) and this yields
\[
\|w_m\|_{L^(B(0,\rho)} \geq \gamma_2 >0.
\]
We can easily check that $w_m$ tends to some $w$ weakly in $E^s$, and that $w$ solves (\ref{eq:3}) with $J(w)=c$.

To complete the proof, we must show that (\ref{eq:11}) is true. If not, along a subsequence, $\|u_m\|_{L^2(B(0,\rho)} \to 0$. But then
\begin{align*}
c+o(1) &= \max_{\theta \geq 0} J(\theta u_m) \geq c_\infty + \frac{1}{2} \int_{B(0,\rho)} \left(V(x)-V_\infty \right) |\phi^\infty(u_m)u_m(x)|^2 \, dx \\
&= c_\infty +o(1),
\end{align*}
i.e. $c \geq c_\infty$. The proof is complete under the stronger assumption (\ref{eq:9}).

Suppose now that 
\[
\liminf_{|x| \to +\infty} V(x)=V_\infty.
\]
Pick $\ge>0$ so that
\[
\liminf_{|x| \to +\infty} V(x) > V_\infty - \ge.
\]
We can apply the previous proof to the potential $V_\ge = V-\ge$: hence either $c$ is larger that the mountain-pass level for this new potential $V_\ge$, or $c$ is a critical value for $J$. In the first case, we conclude by letting $\ge \to 0$ and exploiting the continuity of the mountain-pass levels, Proposition \ref{prop:4.4}.
\end{proof}
\begin{remark}\label{rem:5.2}
The autonomous problem (\ref{eq:10}) was studied in \cite{FQT}, where it is shown that $J^\infty$ has a critical point of mountain-pass type.
When a solution $u \in W^{s,2}(\mathbb{R}^N)$ decays sufficiently fast at infinity and $f$ is independent of $x$, it is possible to prove (conjectured in \cite{DPV} and proved in the recent preprint \cite{Ros-Oton:2012uq}) that the following Pohozaev identity holds:
\[
\left( \frac{N}{2}-s \right) \|u\|^2 = N\int_{\mathbb{R}^N} F(u).
\]
We believe that the results of \cite{FQT} might be improved by using the same ideas of \cite{AzzPomp,PomSec}, in which a natural constraint is built by means of the former variational identity; we will investigate this direction in a forthcoming paper.
\end{remark}

The following is a typical existence result based on the previous Theorem.
\begin{theorem}
Assume that $f$ does not depend on $x$, and that
\begin{itemize}
\item[(V$_3$)] $\liminf_{|x| \to +\infty} V(x) = V_\infty$,
\item[(V$_4$)] $V \leq V_\infty$, but $V$ is not identically equal to $V_\infty$.
\end{itemize}
Then $c$ is a critical value for $J$.
\end{theorem}
\begin{proof}
If $c$ is not a critical value, then $c \geq c_\infty$. By Remark \ref{rem:5.2}, we can fix a solution $w \in E^s$ of (\ref{eq:9}) of mountain-pass type. Therefore
\[
c_\infty = J^\infty(w) = \max_{\theta \geq 0} J^\infty (\theta w).
\]
If $\theta >0$, then
\[
J^\infty(\theta w) = J(\theta w) + \frac{1}{2} \int \left( V_\infty - V(x) \right) |\theta w(x)|^2 \, dx.
\]
Choosing, as usual, $\theta = \phi(w)$, we have
\begin{align*}
c_\infty &\geq J(\phi(w)w) + \frac{1}{2} \int \left( V_\infty - V(x) \right) |\phi(w) w(x)|^2 \, dx \\
&\geq c + \frac{1}{2} \int \left( V_\infty - V(x) \right) |\phi(w) w(x)|^2 \, dx > c.
\end{align*}
This contradiction shows that $c$ must be a critical level for $J$.
\end{proof}
We conclude with an existence result for a parametric equation. We will need some technicalities for the fractional laplacian introduces in the second Section.

We will study the equation
\begin{equation} \label{eq:12}
\slap u + V(\ge x) u = f(u),
\end{equation}
where $\ge$ is a positive \emph{small} parameter. For convenience, we will write $V_\ge(x)=V(\ge x)$.
\begin{theorem}
Suppose that (V$_1$), (V$_2$), ($f_1$--$f_5$) hold with $f$ independent of $x$, and in addition that
\begin{itemize}
\item[(V$_5$)] $V(0)<V_\infty$.
\end{itemize}
Under these assumptions, there is $\ge_0>0$ such that, for every $0<\ge < \ge_0$, equation (\ref{eq:12}) has a nontrivial solution.
\end{theorem}
\begin{proof}
To highlight the presence of $\ge$, we write
\[
J_\ge (u) = \frac{1}{2} \int |\xi|^{2s} |\hat{u}(\xi)|^2 \, d\xi + \frac{1}{2} \int V_\ge (x) |u(x)|^2 \, dx - \int F(u(x))\, dx,
\]
and let $c_\ge$ be the corresponding mountain-pass level. It is clear that $c_\ge$ has the same characterization as in Lemma \ref{lem:4.2}. In view of Theorem \ref{th:5.1}, we want to exclude the possibility that $c \geq c_\ge$. Assume, by contradiction, that this inequality holds true, and fix, as before, a solution $w$ of equation (\ref{eq:10}). Pick a smooth function $\chi \colon [0,+\infty) \to [0,+\infty)$ such that
\begin{align*}
\chi(t) &= 1 \quad\text{if $0 \leq t \leq 1$} \\
|\chi'(t)| &\leq 1 \quad\text{for every $t$}
\end{align*}
and $\operatorname{supp}\chi$ is a compact interval.
For $R>0$, set $\chi_R(t) = \chi(t/R)$, and consider $v = \chi_R w$. By Lemma \ref{lem:Palatucci}, $v \to w$ strongly as $R \to +\infty$.

Given any $\bar{\theta} >0$,
\[
\max_{\theta \geq 0} J^\infty (\theta v) \geq J_\ge (\bar{\theta}v) + \frac{1}{2} \int_{\operatorname{supp} \chi_R} \left( V_\infty - V_\ge(x) \right) |\bar{\theta}v(x)|^2 \, dx.
\]
Choose $\bar{\theta}=\phi_\ge(v)$, where $\phi_\ge(v)$ is defined for $J_\ge$ as $\phi$ was defined for $J$; therefore
\[
\max_{\theta \geq 0} J^\infty (\theta v) \geq c_\ge + \frac{1}{2} \int_{\operatorname{supp} \chi_R} \left( V_\infty - V_\ge(x) \right) |\phi_\ge(v) v(x)|^2 \, dx.
\]
When $\ge \ll 1$, $V_\infty - V_\ge (x) \geq \frac{1}{2} \left( V_\infty - V(0) \right)$ in $\operatorname{supp}\chi_R$, and we get
\[
\max_{\theta \geq 0} J^\infty (\theta v) \geq c_\ge + \frac{1}{4} \left( V_\infty - V(0) \right) \phi_\ge(v)^2 \int_{\operatorname{supp} \chi_R} |v(x)|^2 \, dx.
\]
Recall that $\bar{\theta}=\phi_\ge(v)$ depend on $\ge$ and on $R$, so that we need to bound these quantity in a suitable way.

\medskip

\noindent\emph{Claim 1:} there exists $\theta_0>0$ such that  $\phi_\ge(v) \geq \theta_0$ for $\ge \ll 1$ and $R \gg 1$.

\medskip

\noindent\emph{Claim 2:} there exists a strictly positive function $R \mapsto \psi (R)$ such that $\psi(R) \to 0$ as $R \to +\infty$  and $\max_{\theta \geq 0} J^\infty (\theta v) \leq c_\infty + \psi(R)$.

\medskip

We take these claims for granted, and we finish the proof. Choose $R$ so large that
\[
\int_{\operatorname{supp}\chi_R} |v(x)|^2 \, dx \geq \frac{1}{2} \int |w(x)|^2 \, dx,
\]
which yields
\[
\max_{\theta \geq 0} J^\infty (\theta v) \geq c_\ge + \frac{1}{4} \left( V_\infty - V(0) \right) \theta_0^2 \int |w(x)|^2 \, dx.
\]
Choosing $R$ larger, if needed, we may also assume that 
\[
\psi(R) < \frac{1}{4} \left( V_\infty - V(0) \right) \theta_0^2 \int |w(x)|^2 \, dx;
\]
then $c_\infty > c_\ge$, the desired contradiction.

We now prove the two claims. Recall that $\phi_\ge (v)$ is characterized by the equation
\[
\phi_\ge(v)^2 \left( \int |\xi|^{2s} |\hat{v}(\xi)|^2 \, d\xi + \int V_\ge (x) |v(x)|^2 \, dx \right) = \int f(\psi_\ge(v)v(x)) \phi_\ge(v) v(x)\, dx.
\]
As we did previously, given $\eta>0$, there exists $A_\eta>0$ such that
\[
|f(s)| \leq \eta |s| + A_\eta |s|^{p} \quad\text{for all $s \in \mathbb{R}$}.
\]
Hence
\begin{equation*}
\phi_\ge(v)^2 \left( \int |\xi|^{2s} |\hat{v}(\xi)|^2 \, d\xi + \int V_0 (x) |v(x)|^2 \, dx \right) \leq \int \left( \eta \phi_\ge(v)^2 |v(x)|^2 + A_\eta |\phi_\ge(v)v(x)|^{p+1} \right) dx.
\end{equation*}
By definition,
\[
\int |v(x)|^{p+1} \, dx \leq \int |w(x)|^{p+1}\, dx.
\]
The fractional norm is more delicate to estimate. Recalling that $\chi_R w \to w$ strongly as $R \to +\infty$, we can use Lemma \ref{lem:Palatucci-bis} to deduce that, as $R \to +\infty$,
\[
\|(-\Delta)^{\frac{s}{2}} \left( \chi_R w \right) \|_2 = \|\chi_R (-\Delta)^{\frac{s}{2}} w \|_2 + o(1)
\]
and therefore
\[
\|(-\Delta)^{\frac{s}{2}} v \|_2 \geq \frac{1}{2} \|(-\Delta)^{\frac{s}{2}} w\|_2 +o(1).
\]
Therefore, for $R$ large enough,
\[
\int |\xi|^{2s} |\hat{v}(\xi)|^2 \, d\xi + \int \frac{V_0}{2} |v(x)|^2 \, dx \geq \frac{1}{2} \int |\xi|^{2s} |\hat{w}(\xi)|^2 \, d\xi +  \int \frac{V_0}{4} |w(x)|^2 \, dx
\]
We are ready to conclude: specialize $\eta = V_0/2$ so that
\[
\phi_\ge (v) \geq \left( \frac{\frac{1}{2} \int |\xi|^{2s} |\hat{w}(\xi)|^2 \, d\xi + \int \frac{V_0}{4} |w(x)|^2 \, dx}{A_\eta \int |w(x)|^{p+1}\, dx} \right)^{\frac{1}{p+1}},
\]
and call $\theta_0$ the right-hand side of the last estimate.

Finally, since
\[
\max_{\theta \geq 0} J^\infty (\phi^\infty (v)v) = c_\infty + J^\infty (\phi^\infty(v)v) - J^\infty (w)
\]
But $\chi_R w \to w$ in $E^s$ as $R \to +\infty$, so that $\phi^\infty(v) \to \phi^\infty(w)$. But $\phi^\infty(w)=1$, since $w$ solves (\ref{eq:9}); as a consequence
\[
J^\infty (\phi^\infty(v)v) - J^\infty (w) \to 0
\]
as $R \to +\infty$, and alsco Claim 2 is proved.
\end{proof}
\begin{remark}
Since $\slap$ is a pseudo-differential operator whose symbol is $|\xi|^{2s}$, it is easy to check that it scales as $\slap \to \ge^{2s} \slap$ under the change of variable $x \to \ge x$. This agrees with the usual scaling property for the local laplacian, i.e. $s=1$. Therefore (\ref{eq:12}) is a rescaling of the \emph{singularly perturbed equation}
\[
\ge^{2s} \slap v + V(x) v = f(u).
\]
Our previous theorem provides a (classical) solution when $\ge$ becomes small, but we cannot say that this solution \emph{concentrates} at some point as $\ge \to 0$. Single-peak (and multi-peak) solutions for fractional Schr\"{o}dinger equations are a very stimulating problem, but standard techniques that were developed for the local laplacian do not work out-of-the-box. Roughly speaking, these techniques heavily rely on blow-up and local estimates, and quite often need fine properties of solutions to the limiting problem at $\ge=0$. Non-degeneracy, information on the Morse index, and even uniqueness for the limiting equation are essentially still unknown in the fractional setting; the concept itself of \emph{spike} may need some explanation, since $\slap$ may ``kill'' bumps by averaging on the whole space.

From a more technical viewpoint, a major difficulty is that bound states of (\ref{eq:9}) decay slowly (see Subsection~\ref{subs:5.1}); there is the very little room to make good estimates at infinity.
\end{remark}

\section{Regularity} \label{subs:5.1}

Although we did not mention regularity in our existence results, it is possible to show that the solutions we have found are H\"{o}lder continuous and solve their equation pointwise. Since our equations do not have singularities, the proof of this fact is rather standard.

First of all, recall that the nonlinearity $f$ is sufficiently smooth, and in particular H\"{o}lder continuous. The Sobolev embedding theorem implies that any weak solution $u \in E^s$ of
\[
\slap u + V u = f(x,u)
\]
belongs to some space $L^q$. Then an application of the argument contained in \cite[page 312]{Labutin} shows that, for any ball $B$, $u \in W^{s+\eta,2}(B)$ for some $\eta>0$. Then Proposition 5.1 in \cite{Barrios} implies that $u \in L^\infty(B)$. We can now use the regularity theory of \cite{Cabre:2010fk}, and $u$ turns out to be locally H\"{o}lder continuous. As noticed in \cite{FQT}, moving the ball $B$ around, we get the global H\"{o}lder continuity of the solution $u$.

Finally, we explicitly mention that our solution decay polynomially fast at infinity. Theorem 3.1 of \cite{FQT} can be easily adapted to our setting, so that our solutions $u$ behave at infinity like
\[
\frac{1}{|x|^{N+2s}}.
\]
Of course, this is a specific feature of the fractional laplacian, which does not produce an exponential decay like the ordinary laplacian.

\nocite{*}
\bibliography{global}

    \end{document}